\documentclass[11pt,reqno]{amsart}
\usepackage{amsthm, amsfonts, amssymb, color}
 \usepackage{mathrsfs}
\usepackage{enumerate}
\usepackage{amsmath}
 \usepackage{amstext, amsxtra}
  \usepackage{txfonts}
 \usepackage[colorlinks, linkcolor=black, citecolor=blue, pagebackref, hypertexnames=false]{hyperref}
\usepackage{graphicx}

 \allowdisplaybreaks
\newtheorem{theorem}{Theorem}[section]
\newtheorem{proposition}[theorem]{Proposition}
\newtheorem{corollary}[theorem]{Corollary}

\newtheorem{definition}[theorem]{Definition}

\newtheorem{remark}[theorem]{Remark}

\renewcommand\Re{\operatorname{Re}}
\renewcommand\Im{\operatorname{Im}}

\numberwithin{equation}{section}
\theoremstyle{definition}

\title
 [Limiting absorption principle and spectral multiplier theorems]
{Remarks on $L^p$-limiting absorption principle of Schr\"{o}dinger operators and applications to spectral multiplier theorems}
\author{Shanlin Huang, Xiaohua Yao, and Quan Zheng}

\address{ Shanlin Huang,  School of Mathematics and Statistics, Huazhong University of Science and Technology,  Wuhan,  430074, P.R.China  }
\email{shanlin\_huang@hust.edu.cn}
\address{Xiaohua Yao, Hubei Key Laboratory of Mathematical Sciences and School of Mathematics and Statistics,
 Central China Normal University, Wuhan, 430079, P.R. China}
\email{yaoxiaohua@mail.ccnu.edu.cn }
\address{ Quan Zheng,  School of Mathematics and Statistics, Huazhong University of Science and Technology,  Wuhan,  430074, P.R.China  }
\email{qzheng@hust.edu.cn}

\begin{document}

\maketitle
\begin{abstract} This paper comprises two parts. We first investigate a $L^p$ type of limiting absorption principle for Schr\"{o}dinger operators $H=-\Delta+V$ on $\mathbb{R}^n$ ($n\ge 3$), i.e., we prove the  $\epsilon-$uniform $L^{\frac{2(n+1)}{n+3}}$-$L^{\frac{2(n+1)}{n-1}}$ estimates of the resolvent $(H-\lambda\pm i\epsilon)^{-1}$ for all $\lambda>0$  under the assumptions that the potential $V$ belongs to some integrable spaces and a spectral condition of $H$ at zero is satisfied. As applications, we establish a sharp  H\"ormander type spectral multiplier theorem associated with Schr\"{o}dinger operators $H$ and deduce $L^p$ bounds of the corresponding Bochner-Riesz operators.  Next, we consider the fractional Schr\"{o}dinger operator $H=(-\Delta)^{\alpha}+V$ ($0<2\alpha<n$) and prove a uniform Hardy-Littlewood-Sobolev inequality for $(-\Delta)^{\alpha}$, which generalizes the corresponding result of Kenig-Ruiz-Sogge \cite{KRS}.
\end{abstract}

%\maketitle
\section{Introduction and main results}
\vskip0.3cm
In this paper, we investigate $L^p$-estimates for resolvents of Schr\"{o}dinger operators $-\Delta+V$ and devote their applications to related spectral multiplier operators. Besides, we also study similar estimates for generalized Schr\"{o}dinger operators $(-\Delta)^{\alpha}+V$ ($0<2\alpha< n$). Firstly, let us recall that in a paper of Kenig, Ruiz and Sogge \cite{KRS}, it was shown that for $n\ge 3$, there is a constant $C_p$ independent of $\lambda$ such that
\begin{align}\label{1.1}
\sup_{0<\epsilon<1}\|(-\Delta-(\lambda+i\epsilon))^{-1}\|_{L^{p}-L^{p'}}\leq C_p\lambda^{\frac{n}{2}(\frac 1p-\frac1{p'})-1}, ~~~\lambda>0,
\end{align}
where $\frac{2n}{n+2}\leq p\leq \frac {2(n+1)}{n+3}$ and $\frac{1}{p}+\frac{1}{p'}=1$. Moreover, we mention another basic resolvent estimate  for $H=-\Delta+V$ due to Agmon \cite{A}, known as {\it the limiting absorption principle},
%free resolvent estimate which states that for any $\lambda_0>0$ , there is a constant $C_1(\lambda_0)$ depending on $\lambda_0$ such that
%\begin{align}\label{1.2}
%\sup_{0<\epsilon<1}\|(-\Delta-(\lambda^2+i\epsilon))^{-1}\|_{L^{2, \sigma}-L^{2, -\sigma}}\leq C_1(\lambda_0)\lambda^{-1}, ~~~\lambda>\lambda_0,
%\end{align}
%where  (see e.g. \cite{KJ}). Using \eqref{1.2},  Agmon in his fundamental work \cite{A} firstly established the same type of estimates for Schr\"{o}dinger operators $H=-\Delta+V$, known as {\it the limiting absorption principle},
\begin{align}\label{1.3}
\sup_{0<\epsilon<1}\|(-\Delta+V-(\lambda+i\epsilon))^{-1}\|_{L^{2, \sigma}-L^{2, -\sigma}}\leq C(\lambda_0)\lambda^{-1}, ~~~\lambda>\lambda_0>0,
\end{align}
where  $|V(x)|\leq C(1+|x|)^{-1-}$ and $\sigma>\frac12$ ( $L^{2, \sigma}$ is the usual weighted $L^2$ space ).  Motivated by \eqref{1.1} and \eqref{1.3}, Goldberg and Schlag \cite{GS} showed a $L^p$ version of the limiting absorption principle for the three dimensional Schr\"{o}dinger operators. More specifically, for any given $\lambda_0>0$, they proved
\begin{align}\label{1.4}
\sup_{0<\epsilon<1}\|(-\Delta+V-(\lambda+i\epsilon))^{-1}\|_{L^{\frac43}(\mathbb{R}^3)-L^4(\mathbb{R}^3)}\leq C(\lambda_0)\lambda^{-\frac14}, ~~~\lambda>\lambda_0,
\end{align}
where  $V\in L^p(\mathbb{R}^3) \cap L^{\frac32}(\mathbb{R}^3)$, $p>\frac32$.

In the first part, we will address the problem that whether one can take $\lambda_0=0$ in \eqref{1.4}. This is natural when comparing \eqref{1.4} to the free case \eqref{1.1}, and it's also inspired by spectral multiplier applications (see Section 2.2 below). To this end, we mention that as far as  dispersive estimates for Schr\"{o}dinger equations are concerned, zero is often assumed to be neither a eigenvalue nor a resonance in order to obtain sharp decay estimates (see e.g. \cite{JK, JSS, RS04,G}). In our $L^p$ case, we need to introduce a similar condition as well.
%one usually needs to handle the expansion of the resolvent of $H$ at zero in certain weighted spaces
%(see Definition \ref{def1.1} below)  Furthermore, we shall point out that if assuming that $V\ge 0$, then one can verify this condition  by exploiting the "positivity" of $H$.
\begin{definition}\label{def1.1}
{\it We say that zero is regular with respect to $H=-\Delta+V$ in $L^p(\mathbb{R}^n)$ with some $p\ge 1$, if for any $u\in L^p(\mathbb{R}^n)$ which satisfies $u=-(-\Delta)^{-1}Vu$, then $u\equiv 0$.}
\end{definition}
Equivalently, if zero is a regular point of $H=-\Delta+V$, then the equation $-\Delta u+V u=0$ only  has the trivial solution $u=0$ in $L^p$ space.
Under suitable size conditions on potential $V$, we can prove that zero is always regular if $V\ge 0$ (see section 2.1 below).  On the other hand,  counterexamples can be constructed when $V< 0$. In fact, let $u(x)=(1+|x|^2)^{-\frac{n-2}{2}}$, then it's easy to check that $u\in L^\frac{2(n+1)}{n-1}(\mathbb{R}^n)$. A direct computation shows
$$
\frac{\Delta u}{u}=-\frac{n(n-2)}{(1+|x|^2)^2}<0.
$$
If we put $V= \frac{\Delta u}{u}$, then $V\in L^p(\mathbb{R}^n)$ for all $p\ge \frac {n}{2}$ and $u$ satisfies that $u=-(-\Delta)^{-1}V u$.

 Under this assumption for general potential $V$, we can state our first main result.
\begin{theorem}\label{thm1.2}
Let $n\ge 3$ and $V\in L^{\frac{n}{2}+\sigma} \cap L^{\frac{n}{2}}$, $\sigma>0$ be real-valued. If zero is regular with respect to $H=-\Delta+V$ in $L^\frac{2(n+1)}{n-1}$, then
\begin{align}\label{1.6}
\sup_{0<\epsilon<1}\|(-\Delta+V-(\lambda+i\epsilon))^{-1}\|_{L^{\frac{2(n+1)}{n+3}}-L^\frac{2(n+1)}{n-1}}\leq C\lambda^{-\frac{1}{n+1}}, ~~~ \lambda>0.
\end{align}
\end{theorem}

Theorem \ref{thm1.2}  extends estimates \eqref{1.4} of Goldberg and Schlag \cite{GS} to the limiting case $\lambda_0=0$ and also deals with the higher dimension case. The potential class is almost critical and it relies heavily on results of absence of imbedded eigenvalues of $-\Delta+V$ due to Ionescu and Jerison \cite{IJ}. We also mention that when $V(x)=c/{|x|^2}$, where $c \ge -\frac{(n-2)^2}{4}$, which denotes the best constant in Hardy inequality, the uniform resolvent estimates of the form $L^{p,2}- L^{p', 2}$ in Lorentz space were obtained very recently by Mizutani \cite{Mi}.

As mentioned above, one motivation behind proving uniform resolvent estimates \eqref{1.6} is that they are closely related to the theory of spectral multipliers. Actualy, note that \eqref{1.1} implies the following sharp estimates of the spectral measure associated with the Laplace operator
\begin{align}\label{1.41}
\|dE_{-\Delta}(\lambda)\|_{L^{\frac{2(n+1)}{n+3}}(\mathbb{R}^n)-L^{\frac{2(n+1)}{n-1}}(\mathbb{R}^n)}\leq C\lambda^{-\frac{1}{n+1}},~~~\lambda>0,
\end{align}
 by the Stone formula  $$
 dE_{-\Delta}(\lambda)=\frac{1}{2\pi i}\Big((-\Delta-(\lambda+i0))^{-1}-(-\Delta-(\lambda-i0))^{-1}\Big),\ \ \lambda>0.
 $$For the perturbed case $H=-\Delta+V$, similar results can also be established for $H$ via Theorem \ref{thm1.2}. In particular, we shall show the following result.
\begin{theorem}\label{thm1.4}
Let $n\ge 3$, $H=-\Delta+V$  and $0\leq V\in L^{\frac n2}\cap L^{\frac{n}{2}+\sigma}$ for some $\sigma>0$. Then
\begin{align}\label{1.7}
\|dE_{H}(\lambda)\|_{L^{p}-L^{p'}}\leq C\lambda^{\frac{n}{2}(\frac 1p-\frac 1{p'})-1},~~~\lambda>0,
\end{align}
for all $1\leq p\leq \frac{2(n+1)}{n+3}$.
\end{theorem}

The proof will be given in Section 2, which relies on the fact that zero is regular in $L^\frac{2(n+1)}{n-1}$ when $V\ge 0$.  The sharpness of \eqref{1.41} indicates that  the range of  $p$ in estimates \eqref{1.7} of $dE_H(\lambda)$ is also sharp. Note that in  Chen, et al \cite[Section 7]{COSY}, a smaller range $1\le p\leq \frac{2n}{n+2}$ was obtained based on  dispersive estimates for $e^{itH}$. On the other hand, in Sikora, Yan and the second author \cite{SYY}, non-negative potentials with small enough $L^p$ norm are required in order to obtain the optimal range of $p$.

In section \ref{sec2}, we shall apply Theorem \ref{thm1.4} to establish  H\"ormander-type spectral multiplier theorems, which devote to $L^p$ estimates of a spectral operator $F(H)$ initially defined in $L^2$ through the functional calculus $$F(H)=\int F(\lambda)dE_H(\lambda).$$
The connection with $L^p$ bounds of Bochner-Riesz means associated with $H$ is also discussed, see Theorem \ref{thm1.5} and Corollary \ref{cor1.6}.

The second part of this paper is devoted to extend \eqref{1.4} to the fractional Schr\"{o}dinger operators $H=(-\Delta)^{\alpha}+V$, where $0<2\alpha<n$. Thus it's natural to first prove uniform estimates for the fractional Laplacian  $(-\Delta)^{\alpha}$. We note that besides estimates \eqref{1.1}, the following uniform Sobolev estimates obtained in Kenig, Ruiz and Sogge \cite{KRS} are also closely related.
\begin{align}\label{1.8}
\|u\|_{L^q(\mathbb{R}^n)}\leq C_{p, q}\|(\Delta+z)u\|_{L^p(\mathbb{R}^n)},\, u\in C_0^{\infty}(\mathbb{R}^n),\, z\in \mathbb{C},
\end{align}
where
\begin{align}\label{1.9}
\frac1p-\frac1q=\frac2n \,~~ \text{and}\,~~ \min\left(\ \Big|\frac1p-\frac12\Big|,\Big|\frac1q-\frac12\Big|\ \right)> \frac{1}{2n}.
\end{align}
\eqref{1.8} was originally motivated by certain unique continuation problems for Schr\"{o}dinger operators $-\Delta+V$, which turned out to be connected with many other problems as well, see e.g. \cite{F, F15, FS} for applications of estimating the eigenvalue bounds of Schr\"{o}dinger operators. Hence it's of independent interest to extend \eqref{1.8} to more general situations. Indeed, there are lots of work concerning various generalizations on manifolds, see e.g. \cite{BSSY, DKS, GH, HS, KU,SY,Shen} and the references therein. Recently Sikora, Yan and the second author \cite{SYY} proved uniform Sobolev estimates for real homogeneous elliptic operators under a non-degenerate condition. In this paper, we shall prove the following theorem.
\vskip0.3cm
\begin{theorem}\label{thm1.7}
Let $n\geq 3$, if $\frac{n}{n+1}\leq \alpha<\frac n2$ and  $1<p<q<\infty$ are Lebesgue exponents satisfying
\begin{equation}\label{1.10}
\frac1p-\frac1q=\frac{2\alpha}{n} ~~\text{and}~~ \min\left(\ \Big|\frac1p-\frac12\Big|,\Big|\frac1q-\frac12\Big|\ \right)> \frac{1}{2n},
\end{equation}
then there is a uniform constant $C_{p,q}<\infty$, such that for all $z\in \Bbb{C}$,
\begin{equation}\label{1.11}
\|u\|_{L^q(\mathbb{R}^n)}\leq C_{p,q}\|((-\Delta)^{\alpha}-z)u\|_{L^p(\mathbb{R}^n)},\quad u\in C_{0}^{\infty}(\mathbb{R}^n).
\end{equation}
On the other hand, if $0<\alpha<\frac{n}{n+1}$, then no such uniform estimates exist.
\end{theorem}
%\vskip0.3cm
We note that if $z=0$, then \eqref{1.11} becomes the classical Hardy-Littlewood-Sobolev inequality which is true for  $0<2\alpha<n$ and $\frac1p-\frac1q=\frac{2\alpha}{n}$, $1<p<q<\infty$. Our proof relies heavily on the relation between $((-\Delta)^{\alpha}-z)^{-1}$ and the special case $(-\Delta-z)^{-1}$ ( see \eqref{2.1} and \eqref{2.8} in Section \ref{sec3}), whose kernel can be written explicitly. Based on this observation, we are able to obtain  these estimates off the dual line $\frac1p +\frac1q=1$ by essentially using Stein's oscillatory integral theorem (see e.g. \cite[Lemma 2.2]{KRS}).
%A generalization of $L^p-L^{p'}$ estimates \eqref{1.1} to the fractional case is also obtained, see Corollary \ref{prop2.1}.

The rest of the paper is organized as follows. In section \ref{sec2}, we shall prove Theorem \ref{thm1.2} and Theorem \ref{thm1.4}.  As applications, a sharp spectral multiplier theorem and $L^p$ bound of Bochner-Riesz means associated with $H$ are given. Section \ref{sec3} is devoted to prove Theorem \ref{thm1.7} and related estimates for $H=(-\Delta)^{\alpha}+V$. Throughout the paper, $C$ and $C_j$ denote absolute positive constants whose dependence will be specified whenever necessary. The value of $C$ may vary from line to line.

\section{Limiting absorption principle and spectral multiplier estimates}\label{sec2}

\subsection{$L^p$-limiting absorption principle.} In this subsection, we will prove Theorem \ref{thm1.2} and Theorem \ref{thm1.4}.

{\bf Proof of Theorem \ref{thm1.2}.}
The proof is a perturbative approach, which is based on the  resolvent identity that for $\epsilon> 0$ and $\lambda>0$,
\begin{equation}\label{Reso-formula}R_V(\lambda+i\epsilon)=(I+R_0(\lambda+i\epsilon)V)^{-1}R_0(\lambda+i\epsilon),
\end{equation}
where $R_V(\lambda+i\epsilon)=(H-\lambda-i\epsilon)^{-1}$, $H=-\Delta+V$ and $R_0(\lambda+i\epsilon)$ denotes the resolvent of the free Laplacian. Since our goal is to prove the  $\epsilon-$uniform $L^{\frac{2(n+1)}{n+3}}$-$L^{\frac{2(n+1)}{n-1}}$ estimates of $R_{V}(\lambda+i\epsilon)$, it then suffices to  show that the inverse $(I+R_0(z)V)^{-1}$ exists as a bounded operator on $L^\frac{2(n+1)}{n-1}$, and its norm is uniformly bounded for  $z\in \mathbb{C_+}$, where $\mathbb{C_+}$ denotes the closed upper half plane $\mathbb{C}$. Then Theorem \ref{thm1.2} will follow by combining (\ref{Reso-formula}) and resolvent estimates of the free Laplacian, i.e., the estimates \eqref{1.1}. Now we divide it into the following three steps.

{\bf Step 1 (Existence of $(I+R_0(z)V)^{-1}$ on $L^\frac{2(n+1)}{n-1}$)}

It follows from Lemma 3.1 and Lemma 3.2 in \cite{GS} that $R_0(z)V$ is a compact operator on $L^\frac{2(n+1)}{n-1}$ for $z\in \mathbb{C_+}$ and the inverse $(I+R_0(z)V)^{-1}$ exists and is bounded on $L^\frac{2(n+1)}{n-1}$ for $z\in \mathbb{C_+}\setminus \{0\}$. Hence it remains to prove the case $z=0$.  Note that zero is assumed to be regular in $L^\frac{2(n+1)}{n-1}$, then the equation $(I+(-\Delta)^{-1}V)u=0$ only has the trivial  solution $u=0$ in $L^\frac{2(n+1)}{n-1}$. Hence Fredholm theorem implies that $(I+R_0(z)V)^{-1}$ is also bounded on $L^\frac{2(n+1)}{n-1}$ for the $z=0$.
%and hence for all $\lambda\in \mathbb{R}$, $\mu\in \mathbb{R}$

{\bf Step 2 (The continuity of the map $z\in \mathbb{C_+} \mapsto R_0(z)V$)}

We will prove the map $z\mapsto R_0(z)V$ is continuous in uniform operator topology from the domain $z\in \mathbb{C_+}$ to the space of bounded operators on $L^\frac{2(n+1)}{n-1}$. For convenience, it suffices to prove instead that $T(\lambda)=R_0(\lambda^2)V$ is continuous for $\lambda$ in the first quadrant, i.e.,  $Re \lambda, Im\ \lambda\ge 0$. Suppose that $V$ is bounded, and supported in the ball $\{|x|\leq R\}$. First we consider the case $\lambda\ne0$. We set $0<|\lambda-\mu|<\min(\frac{|\lambda|}{2}, \frac{1}{2R})$ (which implies $\lambda\neq 0$). Note
that the kernel of $(-\Delta-\lambda^2)^{-1}$ satisfies $R_0(\lambda^2)(x)=|x|^{2-n}F(\lambda|x|)$, where $F(z)=z^{\frac{n-2}{2}}K_{\frac{n-2}{2}}(z)$ (see e.g. \cite[p. 338]{KRS}) and $K_\frac{n-2}{2}(z)$ denotes the modified Bessel function of the second kind. Using the fact that $|F(z)|, |F'(z)|\leq C(1+|z|)^{\frac{n-3}{2}}$, one has
\begin{equation}\nonumber
|R_0(\lambda^2)-R_0(\mu^2)(x,y)|\leq
\begin{cases}
|\lambda-\mu||x-y|^{3-n}, \quad \text{if}\quad |x-y|\leq |\lambda|^{-1},\\[4pt]
|\lambda-\mu||\lambda|^{\frac{n-3}{2}}|x-y|^{\frac{3-n}{2}}, \,\text{if} \, |\lambda|^{-1}\leq |x-y|\leq |\lambda-\mu|^{-1},\\[4pt]
|\lambda|^{\frac{n-3}{2}}|x-y|^{-\frac{n-1}{2}}\quad \text{if}\quad |x-y|\ge |\lambda-\mu|^{-1}.
\end{cases}
\end{equation}
Then a direct calculation yields that there exist constants $C_1, C_2$, depending on $\|V\|_{L^{\infty}}$ and $R$, such that
\begin{align*}
\|(R_0(\lambda^2)-R_0(\mu^2))Vf\|_{L^\frac{2(n+1)}{n-1}}&\leq \\
&(C_1(V,R)|\lambda-\mu|+C_2(V,R)|\lambda-\mu|^{\frac{n-1}{2(n+1)}})\|f\|_{L^\frac{2(n+1)}{n-1}}.
\end{align*}

Next we prove the continuity at the origin. Let $|\lambda|<\frac{1}{2R}$, we have
\begin{equation}\nonumber
|R_0(\lambda^2)-R_0(0)(x,y)|\leq
\begin{cases}
|\lambda||x-y|^{3-n}, \quad \text{if}\quad |x-y|\leq |\lambda|^{-1},\\[4pt]
|\lambda|^{\frac{n-3}{2}}|x-y|^{-\frac{n-1}{2}},\quad \text{if}\quad |x-y|\ge |\lambda|^{-1},
\end{cases}
\end{equation}
For the region $|x|>|\lambda|^{-1}$, one has
\begin{align*}
&\|\chi(|x|>|\lambda|^{-1})(R_0(\lambda^2)-R_0(\mu^2))Vf(x)\|_{L^{\frac{2(n+1)}{n-1}}}\\
&\leq C(V,R)\lambda^{\frac{n-3}{2}+\frac{n-1}{2(n+1)}}\|f\|_{L^{\frac{2(n+1)}{n-1}}},
\end{align*}
where $\chi$ denotes the characteristic function. For the region $|x|\leq|\lambda|^{-1}$, if $n=3$, H\"{o}lder's inequality indicates
$$
\|\chi(|x|\leq|\lambda|^{-1})(R_0(\lambda^2)-R_0(\mu^2))Vf(x)\|_{L^4}\leq C(V,R)\lambda^{\frac{1}{4}}\|f\|_{L^4},
$$
if $n>3$, notice that one can choose $r<\min(\frac{n}{n-3}, \frac{2(n+1)}{n-1})$ such that $n-2-\frac{n}{r}>0$ and $|x|^{3-n}\in L_{loc}^r(\mathbb{R}^n)$, then apply Young's inequality we get
\begin{align*}
&\|\chi(|x|\leq|\lambda|^{-1})(R_0(\lambda^2)-R_0(\mu^2))Vf(x)\|_{L^{\frac{2(n+1)}{n-1}}}\\
&\leq \lambda \||x|^{3-n}\chi(|x|\leq |\lambda|^{-1})\ast Vf\|_{L^{\frac{2(n+1)}{n-1}}}\\
&\leq C(V,R)\lambda^{n-2-\frac{n}{r}}\|f\|_{L^{\frac{2(n+1)}{n-1}}},
\end{align*}
which implies that the map $T(\lambda)$ is also continuous at $\lambda=0$.

In order to pass the arguments above to the general case, we note that for any $V\in L^{n/2}$, there exists a bounded function $\tilde{V}$ with compact support such that $\|V-\tilde{V}\|_{L^{n/2}}<\epsilon$.
Then by observing the uniform Sobolev estimates \eqref{1.11}, we can conclude that there exists a uniform constant $C$ such that $\sup_{\lambda\in \mathbb{C_+} }\|R_0(\lambda^2)(V-\tilde{V})\|_{L^{\frac{2(n+1)}{n-1}}-L^{\frac{2(n+1)}{n-1}}}\leq C\|V-\tilde{V}\|_{L^{n/2}}<\epsilon$. Hence the the continuity can be proved for general $V$ by using the following equality
$$
(R_0(\lambda^2)-R_0(\mu^2))V=R_0(\lambda^2)(V-\tilde{V})+(R_0(\lambda^2)-R_0(\mu^2))\tilde{V}+R_0(\mu^2)(\tilde{V}-V).
$$
%The continuity of the map $T(\lambda):\lambda\to R_0(\lambda^2)V$ for $\lambda\in\mathbb{C}\setminus\{0\}$  with $Im \lambda\ge 0$ holds for the special $V$.  Then the continuity of the map $T(\lambda)$ for the general $V\in L^{n/2}$ follows by approximating V by bounded compactly supported functions.

{\bf Step 3 (Uniform boundedness of the norm $\|(I+R_0(z)V)^{-1}\|$ for $z\in \mathbb{C_+}$})

We have established the existence of $(I+R_0(z)V)^{-1}$, and further showed that  the map $z\mapsto R_0(z)V$ is continuous from the domain $z\in \mathbb{C_+}$ to the space of bounded operators on $L^\frac{2(n+1)}{n-1}$. Note that for any $z,\, z_1\in \mathbb{C_+}$, $z\ne z_1$,  $$(I+R_0(z)V)^{-1}-(I+R_0(z_1)V)^{-1}=(I+R_0(z)V)^{-1}\Big((R_0(z_1)V-R_0(z)V\Big)(I+R_0(z_1)V)^{-1}, $$
  then we can obtain that $(I+R_0(z)V)^{-1}$ is a continuous function of $z$ on $\mathbb{C_+}$. In particular, the norm $\|(I+R_0(\lambda+i\epsilon)V)^{-1}\|$ is uniformly bounded for all $0<\epsilon\le 1$ and $0\leq \lambda\leq r$ with any fixed $r>0$. On the other hand, in view of  H\"{o}lder's inequality, our assumption  $V\in L^p$ ($p>\frac{n}{2}$) yields that $\|V\|_{L^{{2(n+1)}/(n-1)}-L^r}\leq C$, where $\frac 1r-\frac{n-1}{2(n+1)}=\frac 1p<\frac 2n$. In addition, an interpolation between estimates \eqref{1.1} and \eqref{1.11} gives $\|R_0(z)\|_{L^r-L^{{2(n+1)}/(n-1)}}\leq C|z|^{\frac{n}{2p}-1}$. Thus there exists some large enough constant $c_0>0$ such that $\|R_0(z)V\|_{L^{{2(n+1)}/(n-1)}-L^{{2(n+1)}/(n-1)}}< \frac12$ provided $z\in \mathbb{C_+}$ with $|z|\ge c_0$. So the Neumann series expansion directly shows that $$\sup_{|z|\ge c_0}\|(I+R_0(z)V)^{-1}\|_{L^{\frac{2(n+1)}{n-1}}-L^{\frac{2(n+1)}{n-1}}}\leq 2.$$ As a conclusion, we obtain that
$$\sup_{\lambda\in \mathbb{R},\, 0<\epsilon\le 1}\|(I+R_0(\lambda+i\epsilon)V)^{-1}\|_{L^{\frac{2(n+1)}{n-1}}-L^{\frac{2(n+1)}{n-1}}}< \infty,$$
which completes the proof of Theorem \ref{thm1.2}.
$\hfill{} \Box$

\vskip0.3cm
{\bf Proof of  Theorem \ref{thm1.4}.} Note that by the Riesz-Thorin interpolation theorem, it suffices to prove \eqref{1.7} for the endpoint $p=\frac{2(n+1)}{n-1}$ and $p=1$ respectively.

{\bf Step 1 (the case $p=\frac{2(n+1)}{n-1}$)} Recall that  $0\leq V\in L^{\frac n2}\cap L^{\frac{n}{2}+\sigma}$ for some $\sigma>0$, we first show that the assumptions on $V$ implies that $0$ is regular in $L^\frac{2(n+1)}{n-1}$.
% Actually, we can prove that the boundedness of $(I+R_0(\lambda^2)V)^{-1}$ in $L^\frac{2(n+1)}{n-1}$ is valid for $\Re \lambda=0$. Again by Fredholm theorem, it suffices to show that if $f\in L^\frac{2(n+1)}{n-1}$, and
%\begin{align}\label{3.7}
%(I+R_0(\lambda^2)V)f=0, \,~~ \lambda=i\mu,\, \mu\in\mathbb{R},
%\end{align}
%then $f=0$. In fact, observe that from \eqref{3.7}, we have $(-\Delta+\mu^2+V)f=0$ in the sense of distribution. We shall split it into two different cases. When $\lambda\ne 0$, set $g=Vf$, where $V\in L^p$. Then H\"{o}lder's inequality indicates $g\in L^r$ with $\frac1r=\frac{n-1}{2(n+1)}+\frac1p$, it then follows from $f=(\mu^2-\Delta)^{-1}g$ and the Sobolev inequality that $f\in L^2$, one can then apply Kato's inequality (see Reed, Simon \cite[p. 183]{RS}) to obtain
%$$
%\Delta|f|\ge\Re(\mathop{\rm sgn} f\cdot\Delta f)=(\mu^2+V)|f|\ge 0.
%$$
%Hence by standard approximation procedure (see \cite[Theorem \uppercase\expandafter{\romannumeral 10.28}]{RS}), we have $f=0$.
To this end, we may assume that there exists some $f\in L^{\frac{2(n+1)}{n-1}}$ such that  $f=-(-\Delta)^{-1}Vf$. Set $g=Vf$, then H\"{o}lder's inequality indicates $\Delta f=g\in L^r$ with $\frac1r=\frac{n-1}{2(n+1)}+\frac1 p$ for some $p>{n\over 2}$.  Then  from $f=-(-\Delta)^{-1}g$ and  the Sobolev inequality  we can conclude that $f\in L^{\frac{2(n+1)}{n-1}}\cap L^{q}$ for some  $q>\frac{2(n+1)}{n-1}$ by
 $$\frac1q=\frac{n-1}{2(n+1)}+\frac1 p-\frac 2n< \frac{n-1}{2(n+1)}.$$
  Now by repeating this bootstrapping procedure we can obtain  $f\in L^{\frac{2(n+1)}{n-1}}\cap L^{q}$ for all  $q>\frac{2(n+1)}{n-1}$.  On the other hand, note that $V\in L^{\frac n2}$, then H\"{o}lder's inequality also gives $\Delta f=-V f\in L^s$  by $$\frac1s=\frac{n-1}{2(n+1)}+\frac2n< \frac{ n+3} {2(n+1)},$$
which implies that $ s'>\frac{2(n+1)}{n-1}$.  Hence  we can choose $q=s'$ such that $f\in L^{\frac{2(n+1)}{n-1}}\cap L^{s'}$ and the inner product $(\Delta f, f)$ makes sense. Besides, in view of Gagliardo-Nirenberg interpolation inequality, we can conclude that $\nabla f\in L^2$ and $$\|\nabla u\|_{L^2}\le C \|\Delta f\|_{L^{s}}^{1/2}\|f\|_{L^{s'}}^{1/2}<\infty.$$  Thus it follows that
$$
0\leq (Vf, f)=(\Delta f, f)=-\|\nabla f\|^2 \leq 0,
$$
which implies  $f=0$ in this case, i.e zero is regular in $L^{\frac{2(n+1)}{n-1}}$. Hence we are allowed to apply the uniform resolvent estimates \eqref{1.6} from Theorem \ref{thm1.2}. Thus, by Stone's formula
 $$
 dE_{H}(\lambda)f=\frac{1}{2\pi i}((H-(\lambda+i0))^{-1}-(H-(\lambda-i0))^{-1})f,
 $$
the estimate \eqref{1.7} is valid for $p=\frac{2(n+1)}{n+3}$.

{\bf Step 2 (the case $p=1$)} Since under our assumption on the potential $V$, the semigroup $e^{-tH}$ satisfies the following Gaussian estimates
\begin{align}\label{2.02}
|e^{-tH}(x,y)|\leq Ct^{-\frac n2}\exp\left(-\frac{|x-y|^2}{4t}\right)
\end{align}
for some $C>0$, see e.g. \cite{Da}. Hence we obtain that for $1\leq p\leq q\leq \infty$,
\begin{align}\label{2.03}
\|e^{-tH}\|_{L^p-L^q}\leq Ct^{-\frac n2(\frac1p-\frac1q)},\,\,\,~~~t>0.
\end{align}
According to the Laplace transform formula
$$
(1+tH)^{-k}=\frac{1}{\Gamma(k)}\int_0^{\infty}e^{-utH}u^{k-1}du,\,\,~~~t>0,\,\,k>0,
$$
we have
\begin{align}\label{2.04}
\|(1+tH)^{-k}\|_{L^p-L^q}&\leq \frac{1}{\Gamma(k)}\int_0^{\infty}e^{-u}(ut)^{-\frac n2(\frac1p-\frac1q)}u^{k-1}du\nonumber\\
&\leq Ct^{-\frac n2(\frac1p-\frac1q)},
\end{align}
provided $k>\frac n2(\frac1p-\frac1q)$. Note that $dE_{H}(\lambda)=2^{2k}(1+H/\lambda)^{-2k}dE_{H}(\lambda)$. Then combine \eqref{2.04} with $L^{p_0}-L^{p'_0}$ estimate of $dE_{H}(\lambda)$ obtained already in step 1, it follows that (see e.g. \cite[Lemma 3.3]{SYY})
\begin{align}\label{2.05}
\|dE_{H}(\lambda)\|_{L^1-L^{\infty}}&
= 2^{2k}\|(1+H/\lambda)^{-k}dE_{H}(\lambda)(1+H/\lambda)^{-k}\|_{L^1-L^{\infty}}\nonumber\\
&\leq C\|(1+H/\lambda)^{-k}\|_{L^{p'_0}-L^{\infty}}\|dE_{H}(\lambda)\|_{L^{p_0}-L^{p'_0}}\|(1+H/\lambda)^{-k}\|_{L^1-L^{p_0}}\nonumber\\
&\leq C\lambda^{\frac n2-1},\,\,\, \lambda>0.
\end{align}
Therefore the desired estimate \eqref{1.7} is valid for all $1\leq p\leq \frac{2(n+1)}{n+3}$.
%Then it's not hard to show (see \cite[Lemma3.3]{SYY}) the estimate is valid for all $1\leq p\leq \frac{2(n+1)}{n+3}$.
 $\hfill{} \Box$

%\begin{remark}\label{rmk2.1}
%We can  construct counterexamples to show that zero may fail to be regular if we don't assume $V\ge 0$. In fact, let $f(x)=(1+|x|^2)^{-\frac{n-2}{2}}$,  then it's easy to check that $f\in L^\frac{2(n+1)}{n-1}(\mathbb{R}^n)$. A direct computation shows
%$$
%\frac{\Delta f}{f}=-\frac{n(n-2)}{(1+|x|^2)^2}<0.
%$$
% If we put $V= \frac{\Delta f}{f}$, then $V\in L^p(\mathbb{R}^n)$ for all $p\ge \frac {n}{2}$ and $f$ satisfies that $f=-(-\Delta)^{-1}V f$.
%\end{remark}
%%Now, we show that Theorem \ref{thm1.5} is obtained as an application.
%%\vskip0.3cm%
%%\emph{The proof of Theorem \ref{thm1.5}.}

\subsection{Applications to spectral multipliers $f(H)$.}
We start by recalling that via a well known $T^*T$ argument (see e.g. Sogge \cite{So}), spectral measure estimate \eqref{1.41} is in fact equivalent to the following Stein-Tomas theorem.
\begin{align}\label{1.5}
(\int_{S^{n-1}}|\hat{f}|^2\, d\sigma)^{\frac12}\leq C\|f\|_{L^p(\mathbb{R}^n)}, ~~~ 1\leq p\leq \frac{2(n+1)}{n+3}.
\end{align}
Such restriction type of estimates are essentially required to obtain sharp results in the theory of spectral multipliers, we refer to \cite{COSY, SYY0, SYY} and literature therein.

As mentioned before,  Theorem \ref{thm1.4} will be used to obtain certain spectral multiplier estimates.
We remark that there are a great number of works devoted to the $L^p$-theory of spectral multipliers $f(H)$ of non-negative self-adjoint operators. This is  a related  area of harmonic analysis, which has attracted a lot of attention during the last thirty years or so. The literature devoted to the subject is so broad that it is impossible to provide complete and comprehensive bibliography. Therefore we quote only a few recent articles, which are directly related to our study, e.g. see \cite{Blu, COSY, DOS, Heb, KU, SYY0, SYY} and therein references.  Actually, for Schr\"odinger operator $H=-\Delta+V$, once we have the spectral measure estimate \eqref{1.7}, then we have the following conclusion.

\begin{theorem}\label{thm1.5}
Suppose that $n\ge 3$, $H=-\Delta+V$ satisfies the assumptions of Theorem \ref{thm1.4}. For any bounded Borel function $F$ such that $\sup_{t>0}\|\eta F(t\cdot)\|_{W^{\alpha}_2}<\infty$
for some $\alpha>\max\{n(\frac{1}{p}-\frac 12), \frac 12\}$ and $1\leq p \leq \frac{2(n+1)}{n+3}$. Here $\eta\in C_0^{\infty}(0, \infty)$ is an arbitrary non-zero auxiliary function and $\|F\|_{W^{\alpha}_2}=\|(1-\frac{d^2}{dx^2})^{\frac{\alpha}{2}}F\|_{L^2(\mathbb{R})}$. Then the operator $F(H)$ is bounded on $L^r(\mathbb{R}^n)$ for all $p< r< p'$.
Furthermore, if the function $F$ is even and compactly supported, then the operator $F(H)$ is bounded for all $p\le r\le p'$.
%(\romannumeral 2) For any bounded Borel function $F$ such that $\sup_{t>0}\|\eta F(t\cdot)\|_{W^{\alpha}_{\infty}}<\infty$
%for some $\alpha>n(\frac{1}{p}-\frac 12)$. Then the operator $F(H)$ is bounded on $L^r(\mathbb{R}^n)$ for all $p<r<p'$.
\end{theorem}
%\vskip0.3cm
 \begin{proof}
 We first recall that in \cite[Proposition 2.2]{SYY}, it was shown that for a non-negative self-adjoint operator $L$ on $L^2(X)$, where homogeneous space $(X, \mu)$ satisfies the doubling conditions. If $H$ satisfies \eqref{2.02} (only Davies-Gaffney estimates are required), and  estimate \eqref{1.7} is valid for $1\leq p\leq \frac{2(n+1)}{n+3}$, then for any bounded Borel function $F$ such that $\sup_{t>0}\|\eta F(t\cdot)\|_{W^{\alpha}_2}<\infty$
for some $\alpha>\max\{n(\frac{1}{p}-\frac 12), \frac 12\}$ with $1\leq p \leq \frac{2(n+1)}{n+3}$.  Then the operator $F(H)$ is bounded on $L^r(\mathbb{R}^n)$ for all $p< r< p'$. Now consider $H=-\Delta+V$ on   $L^2(\mathbb{R}^n)$, we use again the fact that under the assumption $0\leq V\in L^{1}_{\text{loc}}$, the semigroup $e^{-tH}$ satisfies the Gaussian estimates \eqref{2.02}. Hence the proof is complete by applying \eqref{1.7} in Theorem \ref{thm1.4}.
\end{proof}

%\vskip0.3cm

When $p=1$, Theorem \ref{thm1.5} corresponds to a version of classical H\"{o}rmander Fourier multiplier theorem (see e.g. Stein \cite{St} ), which measures the needed regularities of spectral function $F$ in Sobolev space. The variable parameter $p$ in Theorem \ref{thm1.5} is interesting and important.  A remarkable applied example of spectral multipliers is Bochner-Riesz means. Let's recall that Bochner-Riesz operators of index $\delta$ for a non-negative self-adjoint operator $H$ are defined by
$$
S^{\delta}_{R}(H)=\frac{1}{\Gamma(\delta+1)}\left(1-\frac{H}{R}\right)^{\delta}_+, \,\,\, R>0.
$$
Since $(1-\lambda)_+^\delta\in W^{\alpha}_2$
 if and only if $\delta>\alpha-1/2$. Applying Theorem \ref{thm1.5} above (see also \cite[Coro. 3.2]{COSY}), then we obtain the following result.
\vskip0.3cm
\begin{corollary}\label{cor1.6}
Suppose $H$ satisfies assumptions of Theorem \ref{thm1.5} and $1\leq p\leq \frac{2(n+1)}{n+3}$. Then for $\delta>n(1/p-1/2)-1/2$, we have
$$
\sup_{R>0}\|S^{\delta}_{R}(H)\|_{L^p-L^p}\leq C.
$$
\end{corollary}
\vskip0.3cm
In the free case $H=-\Delta$, it's well known that such Bochner-Riesz mean result is a consequence of Stein-Tomas restriction estimates \eqref{1.5}, see e.g. Stein \cite[p. 390]{St}. Also see Sogge \cite{Sog} for Riesz means on compact manifold. In the present perturbed operator $H=-\Delta+V$, the spectral measure estimates \eqref{1.7} play a similar role as restriction estimates \eqref{1.5} in the proof of Corollary \ref{cor1.6}, as showed in \cite{COSY, SYY0}.
\vskip0.3cm
Finally, we mention that in scattering theory, the following intertwining identity
$$
f(H)P_c=W_{\pm}f(-\Delta)W_{\pm}^*
$$
is valid for a Borel functions $f$, where $P_c$ denotes the projection on the continuous spectrum and $W_{\pm}$ represents the corresponding wave operators. It is well known that Yajima \cite{Y}  showed that $W_{\pm}$ is bounded on $L^p$ ($1\le p\le \infty$), thus based on the identity above, the $L^p$ boundedness of $f(H)P_c(H)$ is reduced to the free case $f(-\Delta)$. However, compared to our Theorem \ref{thm1.5}, this wave operator approach requires quite fast decay condition for $V$ or certain regularity depending on dimension (for example, he needed $|V|\lesssim \langle x \rangle^{-5-\epsilon}$ in three dimension).
%More recently Beceanu \cite{Be} proved                                                                                                                                                                                                                                                   $L^p$ boundedness of wave operators in three dimension for a large class of potentials $V\in B=\{V| \sum_{k\in \mathbb{Z}}2^{k/2}\|\chi_{|x|\in[2^k, 2^{k+1}]}V(x)\|_{L^2}<\infty\}$. Even in this case, it's easy to observe that the space $B$ does not contain our condition in Theorem \ref{thm1.5}. In fact, take $f=|x|^{-2}(\ln{|x|})^{-1}\chi_{[|x|\ge 2]}$, then we have $f\in L^p(\mathbb{R}^3)$ for all $p\ge \frac32$, but $f\notin B$.
% For further study in this direction, we refer to \cite{GR, Y2}.
%$V\in B=\{V| \sum_{k\in \mathbb{Z}}2^{k/2}\|\chi_{|x|\in[2^k, 2^{k+1}]}V(x)\|_{L^2}<\infty\}  with $L^{2, 1/2+\epsilon}\subseteq B\subseteq L^{3/2}$$

\section{Limiting absorption for  fractional Schr\"{o}dinger operators}\label{sec3}

\subsection{Proof of Theorem \ref{thm1.7}  (the fractional Laplacian $(-\Delta)^\alpha$)}
We first prove the case  $0<\alpha<\frac{n}{n+1}$. Suppose
estimate \eqref{1.11} is valid for all $z$, in particular, take $p=\frac{2n}{n+2\alpha}$, $q=p'=\frac{2n}{n-2\alpha}$, and choose $z=1-i\epsilon$ with $\epsilon> 0$, then we have
\begin{align}\label{3.01}
\|((-\Delta)^{\alpha}-1+i\epsilon)^{-1}f\|_{L^{\frac{2n}{n-2\alpha}}}\leq C\|f\|_{L^{\frac{2n}{n+2\alpha}}}.
\end{align}
Observe that
$$
\Im \frac{1}{1-|\xi|^{2\alpha}-i\epsilon}=\frac{\epsilon}{(1-|\xi|^{2\alpha})^2+\epsilon^2},
$$
which converges weakly to $d\sigma_{S^{n-1}}$, the surface measure on the unit sphere $S^{n-1}\subset\mathbb{R}^n$ as $\epsilon\to 0$. Therefore it follows from the standard $TT^*$ arguments that
\begin{align*}
(\int_{S^{n-1}}|\hat{f}|^2\, d\sigma)^{\frac12}\leq C\|f\|_{L^{{2n}/(n+2\alpha)}}, \quad f\in L^{{2n}\over n+2\alpha}.
\end{align*}
However, the assumption $0<\alpha<\frac{n}{n+1}$ implies that $\frac{2n}{n+2\alpha}>\frac{2(n+1)}{n+3}$, which contradicts to the Stein-Tomas restriction  theorem (see \eqref{1.5}).

Now we shall prove that estimate \eqref{1.11} is valid under the assumption \eqref{1.10}, and that $\alpha\ge \frac{n}{n+1}$. By homogeneity consideration and the gap condition $n(\frac1p-\frac1q)=2\alpha$, we can assume from now on that $|z|=1$. Denote by $K$ the Schwartz kernel of the resolvent $((-\Delta)^{\alpha}-z)^{-1}$, it's convenient to write $K=K'+K''$, where $K'(x)=K(x)$, if $|x|\leq 1$, and 0 otherwise.

{\bf Case 1: $\alpha$ is an integer.} Suppose $\alpha=m\in \mathbb{Z}^+$, in this case the resolvent can be expressed as
\begin{equation}\label{2.1}
((-\Delta)^{m}-z)^{-1}f= \frac{1}{mz}\sum_{k=0}^{m-1}z_k(-\Delta-z_k)^{-1}f,\quad f\in C_{0}^{\infty}(\mathbb{R}^n),
\end{equation}
where $z_k=z^{\frac1m}e^{i\frac{2k\pi}{m}} (k=0,1,\ldots m-1$) are the $k$-th root of $z$. To estimate $K'$, we shall need the following
expression of the Green function of $(-\Delta-z_k)^{-1}$ (see e.g. \cite[p. 338]{KRS})
\begin{equation}\label{2.2}
(-\Delta-z_k)^{-1}(x,y)= (\frac{-z_k}{|x-y|^2})^{\frac{n-2}{4}}K_{\frac{n-2}{2}}((-z_k)^{\frac12}|x-y|),
\end{equation}
where $K_\frac{n-2}{2}(z)$ has the following asymptotic expansion
\begin{equation}\label{2.3}
K_\frac{n-2}{2}(z)=\frac{\pi\csc({(n-2)\pi}/{2})}{2}\sum_{j=0}^{m-1}\frac{(z/2)^{2j-\frac n2+1}}{\Gamma(j-\frac n2+2)j!}+o(z^{2m-\frac n2-1}),\, z\to 0,
\end{equation}
if $n$ is odd, and
\begin{equation}\label{2.4}
K_\frac{n-2}{2}(z)=\frac{1}{2}\sum_{j=0}^{m-1}\frac{(-1)^j(\frac n2-j-2)!}{j!}(\frac z2)^{2j-\frac n2+1}+o(z^{2m-\frac n2-1}),\, z\to 0,
\end{equation}
if $n$ is even, where we have used the fact that in this case $n\ge 2m+2$. We also note that
\begin{equation}\label{2.5}
\sum_{k=0}^{m-1}z_k^{j+1}=
\begin{cases}
0, \quad \text{if}\quad j=0,1,\ldots,m-2,\\[4pt]
m, \quad \text{if} \quad j=m-1.
\end{cases}
\end{equation}
Then if follows from \eqref{2.1}-\eqref{2.5} that
$$
|K'(x)|\leq C|x|^{2m-n},
$$
hence, Hardy-Littlewood-Sobolev inequality yields
\begin{equation}\label{2.6}
\|K'*f\|_{L^q(\mathbb{R}^n)}\leq C\|f\|_{L^p(\mathbb{R}^n)},
\end{equation}
where $\frac1p-\frac1q=\frac{2m}{n}$.
%The estimates for $K''$ is easy.
Thanks to the expression \eqref{2.1}, it follows immediately  from Stein's oscillatory integral theorem in \cite[Lemma 2.4]{KRS} that
\begin{equation}\label{2.7}
\|K''*f\|_{L^q(\mathbb{R}^n)}\leq C\|f\|_{L^p(\mathbb{R}^n)},
\end{equation}
where $(\frac1p,\frac1q)$ lies either on the open line segment $SS'$ or in the interior of the pentagon $SAOA'S'$ in the Figure below.
Combining \eqref{2.6} and \eqref{2.7}, we prove the estimates \eqref{1.11} in this case.
\begin{figure}
\centering\includegraphics{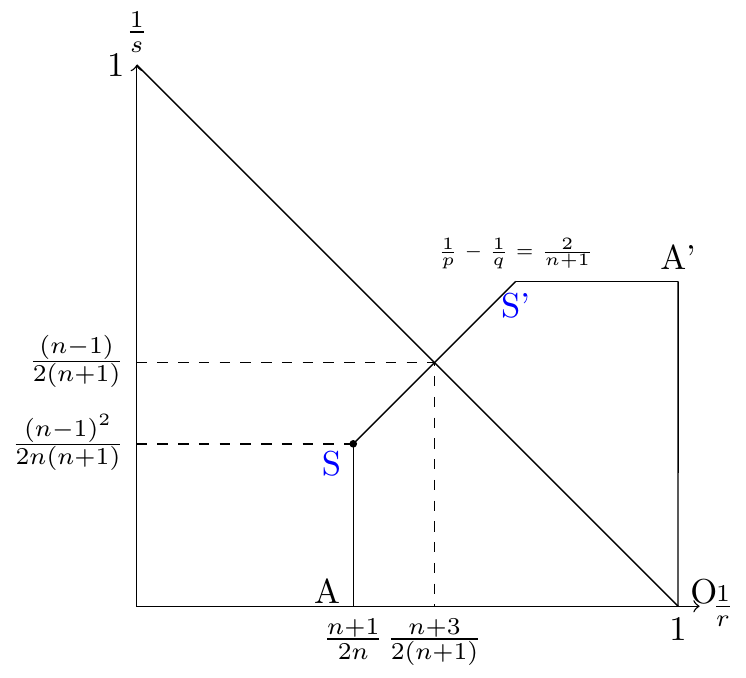}\label{fig1}
\caption{$L^p-L^q$ estimates}
\end{figure}

{\bf Case 2: $\alpha$ is not a integer.} We shall first prove the particular situation $\frac{n}{n+1}\leq \alpha<1$, then show
that the general case can be obtained after a slight modification.

To simplify matters we point out here that we only need to prove estimates \eqref{2.2} for $z=e^{i\theta}, 0<|\theta|<\frac{\pi\alpha}{2}$,
since there exists a constant depending on $\alpha$ such that
$$
\|(-\Delta)^{\alpha}(e^{i\theta}-(-\Delta)^{\alpha})^{-1}\|_{L^1-L^1}\leq C_{\alpha},\quad  \frac{\pi\alpha}{2}\leq |\theta|\leq \pi,
$$
which follows from the fact that the semigroup $e^{-t(-\Delta)^{\alpha}}$ can be extended to an analytic semigroup on $L^1$ in the region $C=\{z:\Re z>0\}$
 and its operator norm is uniformly bounded in every closed subsector of $C$ (see e.g. Komatsu \cite{K}). Hence we are allowed to  write $z=s^{\alpha}$, where $0<|\arg s|<\pi$. We recall the following expression of the resolvent of the fractional power of the negative Laplacian (see Martinez, Sanz \cite{M} )
\begin{align}\label{2.8}
((-\Delta)^{\alpha}-s^{\alpha})^{-1}&=\frac{s^{1-\alpha}}{\alpha}(-\Delta-s)^{-1} \nonumber \\
&+\frac{\sin{\alpha\pi}}{\pi}\int_{0}^{\infty}\frac{\lambda^{\alpha}(\lambda-\Delta)^{-1}}
{\lambda^{2\alpha}-2\lambda^{\alpha}s^{\alpha}\cos{\alpha\pi}+s^{2\alpha}}\, d\lambda \nonumber\\
&\triangleq R_1+R_2.
\end{align}
And we break the kernel of $R_j$ in the same way of $K$  such that $R_j(x)=K_j'(x)+K_j''(x)$. Hence we have $K'=K_1'+K_2'$ and $K''=K_1''+K_2''$.

First, it's easy to see that $K_1'$ is the "good term" for estimating $K'$. Indeed, from \eqref{2.2}, we have
\begin{equation}\label{2.9}
|K_1'(x)|\leq C|x|^{2-n}\leq C|x|^{2\alpha-n}.
\end{equation}
To estimate $K_2'$, we recall that (see Stein \cite[p. 132]{S})
\begin{equation}\label{2.10}
(\lambda-\Delta)^{-1}(x,y)=C\int_{0}^{\infty}e^{-\delta\lambda-\frac{|x-y|^2}{4\pi\delta}}\delta^{\frac{-n+2}{2}}\, \frac{d\delta}{\delta},\quad \lambda>0.
\end{equation}
And a direct computation yields
\begin{align*}
&\int_{0}^{\infty}\frac{\lambda^{\alpha}(\lambda-\Delta)^{-1}(x,y)}{\lambda^{2\alpha}-2\lambda^{\alpha}s^{\alpha}\cos{\alpha\pi}+s^{2\alpha}}\, d\lambda\\
&=\int_{0}^{\infty}e^{-\frac{|x-y|^2}{4\pi\delta}}\delta^{\frac{-n+2}{2}}\, \frac{d\delta}{\delta}\int_{0}^{\infty}\frac{\lambda^{\alpha}e^{-\delta\lambda}}{\lambda^{2\alpha}-2\lambda^{\alpha}s^{\alpha}\cos{\alpha\pi}+s^{2\alpha}}\, d\lambda\\
&=\int_{0}^{\infty}e^{-\frac{|x-y|^2}{4\pi\delta}}\delta^{\frac{-n+2\alpha}{2}}f_{\alpha}(\delta)\, \frac{d\delta}{\delta}\\
&=C|x-y|^{2\alpha-n}+o(|x-y|^{2\alpha-n}), \quad |x-y|\to 0,
\end{align*}
where $f_{\alpha}(\delta)=\int_{0}^{\infty}\frac{t^{\alpha}e^{-t}}{t^{2\alpha}-2\delta^{\alpha}t^{\alpha}s^{\alpha}\cos{\alpha\pi}+(\delta s)^{2\alpha}}\, dt$, and $\lim_{\delta\to 0}f_{\alpha}(\delta)=\Gamma(1-\alpha)$. The last equality follows by observing that
$$
|x|^{-n+2\alpha}=\frac{(4\pi)^{\frac{-n+2\alpha}{2}}}{\Gamma(\frac{n-2\alpha}{2})}\int_{0}^{\infty}e^{-\frac{|x|^2}{4\pi\delta}}\delta^{\frac{-n+2\alpha}{2}}\, \frac{d\delta}{\delta}.
$$
Thus we have
\begin{equation}\label{2.11}
|K_2'(x)|\leq  C|x|^{2\alpha-n}.
\end{equation}
From \eqref{2.11} and \eqref{2.11}, we obtain estimates \eqref{2.6} with $\frac1p-\frac1q=\frac{2\alpha}{n}$.

Next, we estimate the term $K''$. Since, as indicated before, $K_1''$ satisfies estimates \eqref{2.7} with the same exponents there, it suffices to consider convolution with the kernel $K_2''$, which we shall see is the "good term" in this
case. In fact, we recall that \eqref{2.10} implies that for $\lambda>0$

\begin{equation}\label{2.12}
|(\lambda-\Delta)^{-1}(x,y)|\leq
\left\{\begin{array}{cl}
C\lambda^{\frac{n-2}{4}}|x-y|^{-\frac{n-2}{2}}e^{-\sqrt{\lambda}|x-y|}, \, \sqrt{\lambda}|x-y|>1,\\
C|x-y|^{2-n}, \, \sqrt{\lambda}|x-y|\leq 1.
\end{array}\right.
\end{equation}

Using this we obtain
\begin{align*}
&\int_{0}^{\infty}|\frac{\lambda^{\alpha}(\lambda-\Delta)^{-1}(x,y)}{\lambda^{\alpha}-s^{\alpha}e^{\pm i\pi\alpha}}|\, d\lambda\\
&\leq C|x-y|^{2-n}\int_{0}^{|x-y|^{-2}}|\frac{\lambda^{\alpha}}{\lambda^{\alpha}-s^{\alpha}e^{\pm i\pi\alpha}}|\, d\lambda\\ &+C|x-y|^{-n}\int_{1}^{\infty}\frac{t^{2\alpha+\frac n2}e^{-t}}{|t^{2\alpha}-|x-y|^{2\alpha}s^{\alpha}e^{\pm i\pi\alpha}|}\, dt\\
&\leq C|x-y|^{-n-2\alpha}, \quad \text{if}\,\, |x-y|>1,
\end{align*}
where the last inequality follows from the fact that when $|\arg s^{\alpha}|\leq \frac{\pi\alpha}{2}$, then
$$
|\lambda^{\alpha}-s^{\alpha}e^{\pm i\pi\alpha}|\ge C_{\alpha}, \quad \text{if}\,\, 0<\lambda<1,
$$
and
$$
|t^{2\alpha}-|x-y|^{2\alpha}s^{\alpha}e^{\pm i\pi\alpha}|\geq C_{\alpha}|x-y|^{2\alpha},\quad \text{if}\,\, |x-y|>1, \, t>0.
$$
Hence, Young's inequality gives
\begin{equation}\label{2.13}
\|K''*f\|_{L^q(\mathbb{R}^n)}\leq C\|f\|_{L^p(\mathbb{R}^n)}, \quad 1\leq p< q\leq \infty.
\end{equation}
So $K''$ satisfies estimate \eqref{2.9} with the same exponents there, which prove the case $\alpha<1$.

We are left to show the remaining case where $m<\alpha<m+1, m=1,2,\ldots$, the idea is the same so we just sketch arguments. Instead of using formula \eqref{2.10}, we shall use the following functional formula:
\begin{align}
((-\Delta)^{\alpha}-s^{\alpha_m})^{-1}f &=\frac{s^{1-\alpha_m}}{\alpha_m}((-\Delta)^{m+1}-s)^{-1} \nonumber \\
&+\frac{\sin{\alpha_m\pi}}{\pi}\int_{0}^{\infty}\frac{\lambda^{\alpha_m}(\lambda+(-\Delta)^{m+1})^{-1}f}
{\lambda^{2\alpha_m}-2\lambda^{\alpha_m}s^{\alpha_m}\cos{\alpha_m\pi}+s^{2\alpha_m}}\, d\lambda \tag{3.8'}
\end{align}
here $\alpha_m=\frac{\alpha}{m+1}<1$. When estimating $K'$, we shall replace \eqref{2.10} by
$$
(\lambda+(-\Delta)^{m+1})^{-1}(x,y)=\int_{0}^{\infty}e^{-t}t^{-\frac{n}{2m+2}}F(\frac{|x-y|}{t^{{1}/{2(m+1)}}})\, dt,
$$
where $F(\cdot)\in L^p(\mathbb{R})$, $1\leq p\leq \infty$. It's then easy to see that estimate \eqref{2.8} is valid with $\frac1p-\frac1q=\frac{2\alpha}{n}$.
In order to  estimate $K''$, we shall replace \eqref{2.12} by the following estimate
\begin{align*}
&|(\lambda+(-\Delta)^{m+1})^{-1}(x,y)|
\leq \\
&\left\{\begin{array}{cl}
C\lambda^{\sigma_1(m,n)}|x-y|^{\sigma_2(m,n)}\exp\{-\lambda^{1/2(2m+1)}|x-y|^{(m+1)/(2m+1)}\}, \, \lambda^{1/2m}|x-y|>1,\\
C|x-y|^{2m-n}, \, \lambda^{1/2m}|x-y|\leq 1,
\end{array}\right.
\end{align*}
where $\sigma_1(m,n)=\frac{(4m+1)n+2(m+1)}{4(m+1)(2m+1)}-1, \sigma_2(m,n)=\frac{2m+2-n}{2(2m+1)}$.
This in turn follows from the fact that the heat kernel of the $e^{-t(-\Delta)^{m+1}}$ satisfies
$$
|e^{-t(-\Delta)^{m+1}}(x,y)|\leq Ct^{\frac{n}{2(m+1)}}\exp\left\{-C\frac{|x-y|^{\frac{2(m+1)}{2m+1}}}{t^{\frac 1{2m+1}}}\right\}.
$$
Now the desired estimates for $K''$ follows from the integer case ($\alpha=m+1$) and the above pointwise estimates of the resolvent of $(-\Delta)^{m+1}$.
 $\hfill{} \Box$
\begin{remark}\label{rmk1.1}
We note that recently Cuenin \cite{C}  established a type of $L^p-L^{p'}$ ($\frac1p+\frac{1}{p'}=1$) uniform resolvent estimates for fractional Laplacian in order to study the eigenvalue bounds for $H=(-\Delta)^{\alpha}+V$. However, our approach is complete different from the one in \cite{C}.
\end{remark}

We point it out here that it follows from the proof of Theorem \ref{thm1.7} that we also have the following $L^p-L^{p'}$ estimates, which generalize \eqref{1.1} to the fractional case.
\vskip0.3cm
\begin{corollary}\label{prop2.1}
Let $n\geq 3$, and $\alpha\ge \frac{n}{n+1}$, suppose $\frac{2}{n+1}\leq \frac 1p-\frac 1{p'}\leq\frac{2\alpha}{n}$, then there is a uniform constant $C>0$, such that
\begin{align}\label{2.14}
\|((-\Delta)^{\alpha}-z)^{-1}\|_{L^{p}-L^{p'}}\leq C|z|^{\frac{n}{2\alpha}(\frac1p-\frac{1}{p'})-1},\, z\in \mathbb{C}\setminus\{0\}.
\end{align}
\end{corollary}
\vskip0.3cm
\subsection{Applications to perturbed fractional Schr\"{o}dinger operator.}

First, we give a simple application to obtain uniform resolvent estimates when the $L^p$ norm of the potential is small.

\begin{proposition}\label{prop3.1}
Assume $\frac{2n}{n+1}\leq 2\alpha<n$, $V\in L^{\frac{n}{2\alpha}}(\mathbb{R}^n)$, and let $H=(-\Delta)^{\alpha}+V$. There exists a constant $c_0>0$ such that
if $\|V\|_{L^{\frac{n}{2\alpha}}}\leq c_0$, then
$$
\|(H-z)^{-1}\|_{L^p-L^{p'}}\leq C |z|^{\frac{n}{2\alpha}(\frac1p-\frac{1}{p'})-1},\,\,\, z\in \mathbb{C}\setminus\{0\}
$$
for $\max(\frac{2\alpha}{n}, \frac{n+3}{2(n+1)})<\frac1p\leq\frac{n+2\alpha}{2n}$.
\end{proposition}
\begin{proof}
Note that we can choose $q$ such that $(\frac1p, \frac1q)$ satisfies \eqref{1.10}. So we apply Theorem \ref{thm1.7} and H\"{o}lder's inequality to obtain
$$
\|V((-\Delta)^{\alpha}-z)^{-1}\|_{L^{p}-L^{p}}\leq C\|V\|_{L^{\frac{n}{2\alpha}}},
$$
then one can choose $c_0=\frac{1}{2C}$ to deduce
\begin{align}\label{3.1}
\sup_{|z|>0}\|(I+V((-\Delta)^{\alpha}-z)^{-1})^{-1}\|_{L^{p'}-L^{p'}}\leq 2
\end{align}
Now the Proposition follows by combining \eqref{3.1} and \eqref{2.14}.
\end{proof}
\vskip0.3cm
\begin{proposition}\label{prop3.2}
Let $0<2m<n$, $m\in \mathbb{Z}$. Assume $V\in L^{\frac{n}{2m}}$ such that $\max\{\frac{n+1}{2n}, \frac{2m}{n}\}<\frac{2m}{n}+\frac{1}{p_0}<1$ for some $p_0>\frac{2n}{n-1}$. Consider the
map $A(\lambda)=(\lambda-(-\Delta)^m)^{-1}V$, then we have the following

(\romannumeral 1) $A(\lambda)$ is a compact operator on $L^{p_0}(\mathbb{R}^n)$ for any $\lambda\in \mathbb{C}$;\\

(\romannumeral 2) $A(\lambda)$ is continuous from $\mathbb{C_+}\setminus \{0\}$ to the space of bounded operators on $L^{p_0}(\mathbb{R}^n)$;\\

(\romannumeral 3)There exists some $E\subset \mathbb{R}$ with Lebesgue measure zero such that for any compact subinterval $K\subset \mathbb{R}\setminus E\setminus \{0\}$, we have
\begin{align}\label{3.2}
\sup_{\lambda\in K, 0<\epsilon<1 }\|((-\Delta)^m+V-\lambda-i\epsilon)^{-1}\|_{L^{p_0'}-L^{p_0}}<\infty,
\end{align}
where  $ \frac{n+3}{2(n+1)}<\frac1{p_0'}\leq\frac{n+2m}{2n}$. In particular, $\sigma_{\text{sing}}((-\Delta)^m+V)\subset E$.
\end{proposition}
\vskip0.3cm
\begin{proof}
Note that we can choose $q>1$ such that $(\frac{2m}{n}+\frac{1}{p_0}, \frac1q)$ satisfies \eqref{1.10}, hence Theorem \ref{thm1.7} and our assumption on $V$  imply that $A(\lambda)$ is bounded on $L^{p_0}(\mathbb{R}^n)$ for any $\lambda\in \mathbb{C}$. Without lose of generality, we can assume that $V\in L^{\infty}$ with compact support. Now suppose that $\{f_n\}\rightharpoonup 0$ in $L^{p_0}$. First, we observe that
$$
(1+(-\Delta)^m)A(\lambda)=V+(1-\lambda)A(\lambda),
$$
which yields that $A(\lambda)$ is bounded from $L^{p_0}$ to the Sobolev space $H^{2m, p_0}$, then we can apply Rellich's compactness theorem to choose a subsequence
$f_{n_k}$ such that
\begin{align}\label{3.4}
\chi_{[|x|\leq R]} A(\lambda)f_{n_k}\rightarrow 0, ~~~ \text{in}~~ L^{p_0},
\end{align}
where $\chi$ is the characteristic function. On the other hand, we note that for fixed $\lambda\in \mathbb{C}$, when $|x-y|$ is large enough, we have
\begin{equation*}
|(\lambda-(-\Delta)^m)^{-1}(x,y)|\leq
\left\{\begin{array}{cl}
C|x-y|^{2m-n}, \, \lambda=0,\\
C|x-y|^{-\frac{n-1}{2}}, \, \lambda\ne 0.
\end{array}\right.
\end{equation*}
This in turn implies that
\begin{equation*}
\|\chi_{[|x|\geq R]} A(\lambda)f\|_{L^{p_0}}\leq
\left\{\begin{array}{cl}
CR^{2m-n+\frac{n}{p_0}}\|V\|_{L^{\infty}}\|f\|_{L^{p_0}}, \, \lambda=0,\\
CR^{-\frac{n-1}{2}+\frac{n}{p_0}}\|V\|_{L^{\infty}}\|f\|_{L^{p_0}}, \, \lambda\ne 0.
\end{array}\right.
\end{equation*}
Then for any $\epsilon>0$, our assumption indicates that $p_0>\max\{\frac{2n}{n-1}, \frac{n}{n-2m}\}$, and therefore we are allowed  to choose a large enough constant $R$ such that for all $\|f\|_{L^{p_0}}\leq 1$
\begin{align}\label{3.5}
\chi_{[|x|\geq R]} A(\lambda)f<\epsilon.
\end{align}
Combine \eqref{3.4} and \eqref{3.5} , we prove $(\romannumeral 1)$.

To prove $(\romannumeral 2)$, we may also assume $V\in L^{\infty}$, and $\text{supp}~V\subset \{x\in\mathbb{R}^n, |x|\leq R\}$. Furthermore we note that by \eqref{2.1}, it suffices to prove the case $m=1$, which is contained in the proof of Theorem \ref{thm1.7}.

Having established $(\romannumeral 1)$ and $(\romannumeral 2)$, we apply Fredholm theorem to see that
$$
(I+(\lambda-(-\Delta)^m)^{-1}V))^{-1}: ~~~ L^{p'}\rightarrow L^{p'}
$$
exists on $\mathbb{R}\setminus E\setminus \{0\}$, where $m(E)=0$. Moreover, $(I+(\lambda-(-\Delta)^m)^{-1}V))^{-1}$ is a continuous function of $\lambda\in \mathbb{C}_+\setminus E\setminus \{0\}$. Now the desired estimates \eqref{3.2} follows from resolvent identity and \eqref{2.14}.
\end{proof}

\bigskip

\noindent
{\bf Acknowledgements:} The authors would like to express their sincere gratitude to the reviewing referee for his/her many constructive comments which help us greatly improve the previous version. The second author (X. Yao) would like to thank Avy Soffer for useful discussions about spectrum of Schr\"odinger operator, and also to thank Adam Sikora and Lixin Yan for their cooperations and interests on spectral multipliers. X. Yao was supported by NSFC (Grant No. 11371158), the program for Changjiang Scholars and Innovative Research Team in University (No. IRT13066).

\end{document}